\documentclass[12pt,a4paper,oneside]{article}
\usepackage[T1]{fontenc}
\usepackage{graphicx}
\usepackage{mathtools}
\usepackage{amssymb}
\usepackage{amsthm}
\usepackage{thmtools}
\usepackage{xcolor}
\usepackage{nameref}
\usepackage{hyperref}
\newtheorem{theorem}{Theorem}[section]
\newtheorem{corollary}[theorem]{Corollary}
\newtheorem{lemma}[theorem]{Lemma}
\newtheorem{question}[theorem]{Question}
\theoremstyle{definition}
\newtheorem{definition}[theorem]{Definition}

\begin{document}
	
	\title{ An answer to Goswami's question and new sources of $IP^{\star}$-sets containing combined zigzag structure}

	\date{}
	\author{Pintu Debnath
		\footnote{Department of Mathematics,
			Basirhat College,
			Basirhat-743412, North 24th parganas, West Bengal, India.\hfill\break
			{\tt pintumath1989@gmail.com}}
	}
	\maketitle	
	
	\begin{abstract}
	$A$ set is called $IP$-set in a semigroup $\left(S,\cdot \right)$ if it contains finite products of a sequence. A set that intersects with all $IP$-sets is called $IP^\star$-set. It is a well known and established result by Bergelson and Hindman that if $A$ is an $IP^{\star}$-set,  then for any sequence $\langle x_{n}\rangle_{n=1}^{\infty}$,  there exists a sum subsystem $\langle y_{n}\rangle_{n=1}^{\infty}$ such that $FS\left(\langle y_{n}\rangle_{n=1}^{\infty}\right)\cup FP \left(\langle y_{n}\rangle_{n=1}^{\infty}\right)\subset A$. In \cite[Question 3]{G}, S. Goswami posed the  question: if we replace the single sequence by $l$-sequences, then is it possible to obtain a sum subsystem such that all of its zigzag finite sums and products will be in $A$. Goswami has given   affirmative answers  only for dynamical $IP^{\star}$-sets which are not equivalent to thouse of $IP^{\star}$-sets, but are rather significantly stronger. In this article, we will give the answer to Goswami's question that was unknown until now.
	\end{abstract}

	\textbf{Keywords:} $IP^{\star}$-set, $IP_{r}^{\star}$-set,  Measure preserving system, Algebra
 
 of the  Stone-\v{C}ech compactifications of discrete semigroup.
	
	\textbf{MSC 2020:} 05D10, 22A15, 54D35
	
	\section{Introduction}
		Let $r\in \mathbb{N}$ and  $\mathbb{N}=C_{1}\cup C_{2} \cup \ldots \cup C_{r}$. Then atleast one cell is  an $IP$-set, which is called the Hindman finite sum theorem. The Hindman finite sum theorem was a conjecture of Graham and Rothschild and proved by Hindman in \cite{H}. Given a sequence $\langle x_{n}\rangle_{n=1}^{\infty}$ in $\mathbb{N}$, we say that $\langle y_{n}\rangle_{n=1}^{\infty}$ is a sum subsystem of $\langle x_{n}\rangle_{n=1}^{\infty}$ provided there exists a sequence $\langle H_{n}\rangle_{n=1}^{\infty}$ of non-empty finite subset such that $\max H_{n}<\min H_{n+1}$ and $y_{n}=\sum_{t\in H_{n}}x_{t}$ for each $n\in\mathbb{N}$. In \cite{BH2}  Bergelson and Hindman characterized $IP^{\star}$-sets by introducing the following theorem: 
	
	\begin{theorem}\label{FSP}
		Let $ \langle x_{n}\rangle_{n=1}^{\infty}$ be a sequence in $\mathbb{N}$ and $A$ be  an $IP^{\star}$-set in $\left(\mathbb{N},+\right)$. Then there exists a subsystem $\langle y_{n}\rangle_{n=1}^{\infty}$ of $ \langle x_{n}\rangle_{n=1}^{\infty}$ such that $$FS\left(\langle y_{n}\rangle_{n=1}^{\infty}\right)\cup FP\left(\langle y_{n}\rangle_{n=1}^{\infty}\right)\subseteq A.$$
	\end{theorem}
	
		Now it is essential to  give a brief review of algebraic structure of the Stone-\v{C}ech
	compactification of a semigroup $\left(S,+\right)$, not necessarily commutative with the discrete topology to present some results in this article.

	The set $\{\overline{A}:A\subset S\}$ is a basis for the closed sets
	of $\beta S$. The operation `$+$' on $S$ can be extended to
	the Stone-\v{C}ech compactification $\beta S$ of $S$ so that $(\beta S,+)$
	is a compact right topological semigroup (meaning that for each    $p\in\beta$ S the function $\rho_{p}\left(q\right):\beta S\rightarrow\beta S$ defined by $\rho_{p}\left(q\right)=q+ p$ 
	is continuous) with $S$ contained in its topological center (meaning
	that for any $x\in S$, the function $\lambda_{x}:\beta S\rightarrow\beta S$
	defined by $\lambda_{x}(q)=x+q$ is continuous). This is a famous
	Theorem due to Ellis that if $S$ is a compact right topological semigroup
	then the set of idempotents $E\left(S\right)\neq\emptyset$. A nonempty
	subset $I$ of a semigroup $T$ is called a $\textit{left ideal}$
	of $S$ if $T+I\subset I$, a $\textit{right ideal}$ if $I+T\subset I$,
	and a $\textit{two sided ideal}$ (or simply an $\textit{ideal}$)
	if it is both a left and right ideal. A $\textit{minimal left ideal}$
	is the left ideal that does not contain any proper left ideal. Similarly,
	we can define $\textit{minimal right ideal}$ and $\textit{smallest ideal}$.
	
	Any compact Hausdorff right topological semigroup $T$ has the smallest
	two sided ideal
	
	$$
	\begin{aligned}
		K(T) & =  \bigcup\{L:L\text{ is a minimal left ideal of }T\}\\
		&=  \bigcup\{R:R\text{ is a minimal right ideal of }T\}.
	\end{aligned}$$

	Given a minimal left ideal $L$ and a minimal right ideal $R$, $L\cap R$
	is a group, and in particular contains an idempotent. If $p$ and
	$q$ are idempotents in $T$ we write $p\leq q$ if and only if $p+q=q+p=p$.
	An idempotent is minimal with respect to this relation if and only
	if it is a member of the smallest ideal $K(T)$ of $T$. Given $p,q\in\beta S$
	and $A\subseteq S$, $A\in p+ q$ if and only if the set $\{x\in S:-x+A\in q\}\in p$,
	where $-x+A=\{y\in S:x+ y\in A\}$. See \cite{HS} for
	an elementary introduction to the algebra of $\beta S$ and for any
	unfamiliar details.
	
	Let $A$ be a subset of $S$. Then $A$ is called central if and only if $A\in p$, for some minimal idempotent of $\beta S$ and $A$ is called central${}^{\star}$ if and only if  $A$ intersects with all central sets.

In \cite{D2} and \cite{D1}, D. De,  established a similar type of  the Theorem \ref{FSP} for central${}^{\star} $-set and $C^{\star}$. Where sequences have been consider from the class of minimal sequences and almost minimal sequences.

\begin{definition}
	Let $l\in\mathbb{N}$ and $l$-sequences $\langle x_{n}^{\left(1\right)}\rangle_{n=1}^{\infty},\langle x_{n}^{\left(2\right)}\rangle_{n=1}^{\infty},\ldots,\langle x_{n}^{\left(l\right)}\rangle_{n=1}^{\infty}$ in $\mathbb{N}$
	\begin{itemize}
		\item[(a)] $ZFS\left(\langle x_{n}^{\left(i\right)}\rangle_{i,n=1,1}^{l,\infty}\right)$\\
		$=\left\{\sum_{t\in H}y_{t}:H\in\mathcal{P}_{f}\left(\mathbb{N}\right) \text{ and } y_{i}\in\left\{ x_{i}^{\left(1\right)},x_{i}^{\left(2\right)},\ldots,x_{i}^{\left(l\right)}\right\}	\text{ for any } i\in \mathbb{N}\right\}$.
		\item[b] $ZFP\left(\langle x_{n}^{\left(i\right)}\rangle_{i,n=1,1}^{l,\infty}\right)$\\
		$=\left\{\prod_{t\in H}y_{t}:H\in\mathcal{P}_{f}\left(\mathbb{N}\right) \text{ and } y_{i}\in\left\{ x_{i}^{\left(1\right)},x_{i}^{\left(2\right)},\ldots,x_{i}^{\left(l\right)}\right\}	\text{ for any } i\in \mathbb{N}\right\}$.
	\end{itemize}
\end{definition}

In \cite[Question 3]{G}, S. Goswami asked the following question associated with $l$-many sequences instead of single sequence: 

\begin{question}\label{Goswami's Question}
	Let $l\in\mathbb{N}$ and $A\subseteq\mathbb{N}$ is an $IP^{\star}$ set in $\left(\mathbb{N},+\right)$. Then for any $l$ sequences $\langle x_{n}^{\left(1\right)}\rangle_{n=1}^{\infty},\langle x_{n}^{\left(2\right)}\rangle_{n=1}^{\infty},\ldots,\langle x_{n}^{\left(l\right)}\rangle_{n=1}^{\infty}$ in $\mathbb{N}$ whether there exists a $l$ sum subsystems $\langle y_{n}^{\left(i\right)}\rangle_{n=1}^{\infty}$ of  $\langle x_{n}^{\left(i\right)}\rangle_{n=1}^{\infty}$ for each $i\in\left\{ 1,2,\ldots,l\right\}$  such that $$ZFS\left(\langle y_{n}^{\left(i\right)}\rangle_{i,n=1,1}^{l,\infty}\right)\bigcup ZFP\left(\langle y_{n}^{\left(i\right)}\rangle_{i,n=1,1}^{l,\infty}\right)\subset A.$$
\end{question}
 
Goswami did not give the answer to the above question,  but he established that the conclusion of the above question is true for dynamical $IP^{\star} $-sets. There is an $IP^{\star}$-set,  which is not dynamical $IP^{\star} $-set is given  in \ref{IP infty set}. Now it is right time to define dynamical $IP^{\star} $-set.

\begin{definition}{\cite[Definition 19.29, page 503]{HS}}
		
	\begin{itemize}
		\item[(a)]  A measure space is a triple $\left(X,\mathcal{B},\mu\right)$ where $X$ is a set, $\mathcal{B}$ is a $\sigma$-algebra of subsets of $X$, and is a countably additive measure on $\mathcal{B}$ with $\mu\left(X\right) $ finite.
		\item[(b)] Given a measure space $\left(X,\mathcal{B},\mu\right) $ a function $ T:X\rightarrow X$ is a measure preserving transformation if and only if for all $B\in\mathcal{B}$, $ T^{-1}\left[B\right]\in\mathcal{B}$ and $\mu\left(T^{-1}\left[B\right]\right)=\mu\left(B\right)$
		\item[(c)] Given a semigroup $S$ and a measure space $\left(X,\mathcal{B},\mu\right)$ a measure preserving action of $S$ on $X$ is an indexed family $ \langle T_{s}\rangle_{s\in S}$ such that each $T_{s}$ is a measure preserving transformation of $X$ and $T_{s}\circ T_{t}=T_{st}$ . It is also required that if $ S$ has an identity$ e$ then $ T_{e} $ is the identity function on $X$.
		\item[(d)] A measure preserving system is a quadruple $ \left(X,\mathcal{B},\mu,\langle T_{s}\rangle_{s\in S}\right) $ such that $ \left(X,\mathcal{B},\mu\right) $ is a measure space and $ \langle T_{s}\rangle_{s\in S}$ is a measure preserving action of $ S$ on $ X $.
	\end{itemize}
	
\end{definition}	
 In this article we consider, $\mu\left(X\right)=1$, i.e., probability measure,  $\left(S,\cdot\right)	= \left(\mathbb{N},+\right)$ with $T_{n}=T^{n} $ and $T_{n}^{-1}=T^{-n}$.
 
  \begin{definition}\label{dynamical IP* set }
 	A subset $C$ of $\mathbb{N}$ is dynamical essential $IP^{\star}$- set iff there exist a measure preserving system $\left(X,\mathcal{B},\mu,T\right)$ and  $A\in\mathcal{B}$ with $\mu\left(A\right)>0$ such that $\left\{ n\in\mathbb{N}:\mu\left(A\cap T^{-n}A\right)>0\right\} \subseteq C$.	
 \end{definition}
 
In \textbf{section 2}, we will provide a negative answer to Goswami's question. Naturally, the  question arises:  for what type of $IP^{\star}$- sets satisfy the conclusion of Goswami's question. The obvious class  is  the family of  dynamical  $IP^{\star}$- sets.
 
In the recent article \cite{LL},  Liang and Liao  have find out  a new  class of $IP^{\star}$-sets (topological dynamical $IP^{\star}$-sets) satisfying  the conclusion of Goswami's question and In \cite{Z1},  T.Zhang proved that for any $IP^{\star}$-set, $A$ and for certain $l$-sequences (compatible sequences), there exists a diagonal sum subsystem such that all of its zigzag finite sums and products are contained in $A$. So, it can be said that  the  negative answer to Goswami's question  amplify the significance of their article \cite{LL}, \cite{Z1}.

In \textbf{section 3}, we will show that certain types of $IP^{\star}$- sets associated with multiple recurrence and mild mixing systems satisfy the conclusion of Goswami's question \cite[Question 3]{G}.
 
	 \section{An answer to Goswami's question}

	 Before answering Goswami's question \ref{Goswami's Question}, we will mention some observations on dynamical $IP^{\star}$-set and recall the definition of dynamical $IP^{\star}$-set.

	 \begin{definition}\label{dynamical IP* set }
	 	A subset $C$ of $\mathbb{N}$ is dynamical essential $IP^{\star}$- set iff there exist a measure preserving system $\left(X,\mathcal{B},\mu,T\right)$ and  $A\in\mathcal{B}$ with $\mu\left(A\right)>0$ such that $\left\{ n\in\mathbb{N}:\mu\left(A\cap T^{-n}A\right)>0\right\} \subseteq C$.	
	 \end{definition}
	 
	 Let $\left(S,+\right)$ be a commutative semigroup and $A\subseteq S$.
	 \begin{itemize}
	 
	 	\item  Let $r\in \mathbb{N}$. The set $A$ is $IP_{r}$-set if and only if  there exists a sequence $\langle x_{n}\rangle _{n=1}^{r}$ in $S$ such that $FS\left(\langle x_{n}\rangle _{n=1}^{r}\right)\subseteq A$. Where $$ FS\left(\langle x_{n}\rangle _{n=1}^{r}\right)=\left\{ \sum_{n\in F}x_{n}:F\subseteq\{1,2,\ldots,r\}\right\}. $$
	 	
	 	\item Let $r\in \mathbb{N}$. The set $A$ is called $IP_{r}^{\star}$-set, when it intersects with all $IP_{r}$-sets.
	 	
	 	\item  The set $A$ is called $IP_{\infty}$-set if for any $r\in \mathbb{N}$, $A$ is an $IP_{r}$-set.
	 	
	 \end{itemize}
	 How much stronger a dynamical  $IP^{\star}$-set is than others $IP^{\star}$-sets  can be understood from the  following two consecutive theorems:
	 
	 \begin{theorem}\label{dynamical IPr*}
	 	Let $C$ be a dynamical $IP^{\star}$-set. Then $C$ is an $IP_{r}^{\star}$-set for some $r\in\mathbb{N}$.
	 \end{theorem}
	 
	\begin{proof}
		As $C$ is an $IP^{\star}$-set in $\mathbb{N}$, there exist a measure preserving system $\left(X,\mathcal{B},\mu,T\right)$  and $A\in\mathcal{B}$ with $\mu\left(A\right)>0$ such that $\left\{ n\in\mathbb{N}:\mu\left(A\cap T^{-n}A\right)>0\right\} \subseteq C$. It is sufficient to prove that $E=\left\{ n\in\mathbb{N}:\mu\left(A\cap T^{-n}A\right)>0\right\}$ is an $IP_{r}^{\star}$-set for some $r\in\mathbb{N}$. Now as $\mu\left(A\right)>0$, by Archimedean property of real numbers, we can find a positive integer $r$ such that $r \mu\left(A\right)>1$. Now take  finite sequence $\left\{x_{n}\right\}_{n=1}^{r}$ in $\mathbb{N}$. Now, atleast two set of $$\left\{ T^{-x_{1}}A,T^{-\left(x_{1}+x_{2}\right)}A,\ldots,T^{-\left(x_{1}+x_{2}+\ldots,+x_{r}\right)}A\right\} $$ is disjoint. Otherwise, $$\mu\left(T^{-x_{1}}A\cup T^{-\left(x_{1}+x_{2}\right)}A\cup\ldots\cup T^{-\left(x_{1}+x_{2}+\ldots,+x_{r}\right)}A\right)>1.$$ Then there exist $i,j\in \left\{1,2,\ldots,r\right\}$ with $i<j$ such that $$\mu\left(T^{-\left(x_{1}+x_{2}+\ldots,+x_{i}\right)}A\cap T^{-\left(x_{1}+x_{2}+\ldots,+x_{j}\right)}A\right)>0,$$ which implies $\mu\left(A\cap T^{-\left(x_{i+1}+x_{i+2}+\ldots,+x_{j}\right)}A\right)>0$. hence $FS\left\{ x_{n}\right\} _{n=1}^{r}\cap E\neq\emptyset$.
	\end{proof}

\begin{theorem}\label{Dynamical IPn*}
Let $k\in\mathbb{N}$.	Let $C$ be a dynamical $IP^{\star}$-set in $\mathbb{N}$. If $p_{1},p_{2},\ldots,p_{k}\in E\left(\beta \mathbb{N}\right)$, then $C\in p_{1}+p_{2}+\ldots+p_{k}$.
\end{theorem}	

\begin{proof}
   	As $C$ is a dynamical $IP^{\star}$-set in $\mathbb{N}$, there exist a measure preserving system $\left(X,\mathcal{B},\mu,T\right)$  and $A\in\mathcal{B}$ with $\mu\left(A\right)>0$ such that $$\left\{ n\in\mathbb{N}:\mu\left(A\cap T^{-n}A\right)>0\right\} \subseteq C.$$ It is sufficient to prove that $B=\left\{ n\in\mathbb{N}:\mu\left(A\cap T^{-n}A\right)>0\right\}\in  p_{1}+p_{2}+\cdots+p_{k} $. Let, $B\in p_{1}+p_{2}+\cdots+p_{k}$. Now $B\in p_{k+1}$ , as $B$ is $IP^{\star}$-set. We will prove $B\subset\left\{ m:-m+B\in p_{n+1}\right\}$. 
   	 If $m\in B$, then $\mu\left(A\cap T^{-m}A\right)>0$. Let $D=A\cap T^{-m}A$ and $\mu \left(D\right)>0$. So, 
   	$\left\{ n:\mu\left(D\cap T^{-n}D\right)>0\right\} $ is $ IP^{\star}$-set, which implies, $ \left\{ n:\mu\left(A\cap T^{-\left(n+m\right)}A\right)>0\right\} $ is $ IP^{\star}$-set. Now $$m+n\in B\implies n\in-m+B\in p_{n+1}.$$
   	 Which implies,  $$\left\{ m:-m+B\in p_{k+1}\right\} \in p_{1}+p_{2}+\cdots+p_{k}$$
   	 and $B\in p_{1}+p_{2}+\cdots+p_{k}.$
   	
\end{proof}
	
 It is clear from the above mentioned theorems that dynamical $IP^{\star}$-sets enjoy some significant properties. In Lemma \ref{IP infty set}, we will show that the set   $ A=\left\{ \sum_{t\in H_{1}}2^{2n}+\sum_{t\in H_{2}}2^{2n+1}:H_{1}<H_{2}\right\} $ is $IP_{\infty}$-set but not $IP$-set.  And as a consequence the set  $B=\mathbb{N}\setminus A$ is $IP^{\star}$ set but not dynamical $IP^{\star}$-set. We know that,  for 	$p_{1},p_{2},\ldots,p_{k}\in E\left(\beta \mathbb{N}\right)$, members of $ p_{1}+p_{2}+\ldots+p_{k}$ was characterized by Bergelson and Hindman in \cite{BH}.

\begin{definition}
	  Let $ n\in\mathbb{N}$, let $ A\subseteq \mathbb{N}$.  Then $A$ is an $ IP^{n} $ set if and only if there exist for each $i\in\left\{ 1,2,\ldots,n\right\} $ a sequence $ \langle x_{i,t}\rangle_{t=1}^{\infty} $ such that $$ \left\{ \sum_{i=1}^{n}\sum_{t\in H_{i}}x_{i,t}:H_{1},H_{2},\ldots,H_{n}\in\mathcal{P}_{f}\left(\mathbb{N}\right)\,\text{and}\,H_{1}<H_{2}<\ldots<H_{n}\right\} \subseteq A $$ Also, $A $ is an $ IP^{n\star} $-set if and only if $A$ has nonempty intersection with every $ IP^{n} $-set in $ \mathbb{N}$.
\end{definition}
 
We get combinatorial descriptions of $IP^{n\star}$-sets from \cite[Theorem 4.1]{BH}.

\begin{theorem}\label{Cha IPn*}
	 Let $ n\in\mathbb{N}$, let $A\subseteq\mathbb{N}$. The following statements are equivalent.
	 \begin{itemize}
	 		\item[(a)] A is an $IP^{n*} $-set.
	 	\item[(b)]  Whenever $\langle x_{1,t}\rangle_{t=1}^{\infty},\langle x_{2,t}\rangle_{t=1}^{\infty},\ldots,\langle x_{n,t}\rangle_{t=1}^{\infty} $ are sequences in $ S $, there exist for each $ i\in\left\{ 1,2,\ldots,n\right\} $ a product subsystem $ \langle y_{i,k}\rangle_{k=1}^{\infty}$ of $ \langle x_{i,t}\rangle_{t=1}^{\infty} $ such that $$ \left\{ \sum_{i=1}^{n}\sum_{t\in H_{i}}y_{i,k}:H_{1},H_{2},\ldots,H_{n}\in\mathcal{P}_{f}\left(\mathbb{N}\right)\,\text{and}\,H_{1}<H_{2}<\ldots<H_{n}\right\} \subseteq A. $$
	 \end{itemize}
	
\end{theorem}

we get the following result from \cite[Corollary 4.4]{BH}

\begin{theorem}
	Let  $A\subset \mathbb{N}$. Then $A$ is an $IP^{n\star } $-set if and only if for all idempotent $p_{1},p_{1},\ldots,p_{n}\in \beta \mathbb{N}$ one has $A\in p_{1}+p_{1}+\ldots+p_{n}$.
\end{theorem}

  Now, let two sequences $\langle x_{n}\rangle_{n=1}^{\infty}$ and $\langle y_{n}\rangle_{n=1}^{\infty}$. Then $ZFS\left( \langle x_{n}\rangle_{n=1}^{\infty}, \langle y_{n}\rangle_{n=1}^{\infty}\right)$ is an  $IP^{2}$-set generated by $\langle x_{n}\rangle_{n=1}^{\infty}$ and $\langle y_{n}\rangle_{n=1}^{\infty}$.\textbf{ By the Theorem \ref{Cha IPn*}, we will be able to give a negative answer to Goswami's question \cite[Question 3]{G} if we can find  an $IP^{\star}$-set, which is not an  $IP^{2\star}$-set}. The existence of the required set is given by the following theorem:
 \begin{theorem}
 	There exists an $IP^{\star}$-set which is not  $IP^{2\star}$-set.
 \end{theorem}
 
 \begin{proof}
 	By \cite[exercise 6.1.4]{HS}, $E\left(\beta \mathbb N \right)$ is not a semigroup of $\beta\mathbb{N}$. Then, we can find two idempotents , $p,q\in E\left( \beta\mathbb{N}\right)$, such that $p+q\notin E\left( \beta\mathbb{N}\right)$. Then there exists a set $A\subset \mathbb{N} $ with $A\in p+q$,  which is not an $IP$-set. As, $A\in p+q$, then $A$ is an $IP^{2}$-set. By the characterization of $IP^{2}$-set, we have, two sequences $ \langle x_{n}\rangle_{n=1}^{\infty}$ and $\langle y_{n}\rangle_{n=1}^{\infty} $ such that $$\left\{ \sum_{t\in H_{1}}x_{t}+\sum_{t\in H_{2}}y_{t}:H_{1},H_{2}\in\mathcal{P}_{f}\left(\mathbb{N}\right)\,\text{and}\,H_{1}<H_{2}\right\} \subseteq A. $$ Then it is clear that $B=\mathbb{N}\setminus A$ is an $IP^{\star}$-set but not $IP^{2\star}$-set.
 	
 \end{proof} 
 
 From the  following theorem, we can find an $IP^{\star}$-set,  which is neither an  $IP^{2\star}$-set nor  an $IP_{r}^{\star}$-set for any $r\in\mathbb{N}$ in a constructive way.
 \begin{theorem}\label{IP infty set}
 	Let two sequences $ \langle2^{2t}\rangle_{t=1}^{\infty}$ and $\langle2^{2t+1}\rangle_{t=1}^{\infty} $ in $ \mathbb{N}$ . then $$A= \left\{ \sum_{t\in H_{1}}2^{2n}+\sum_{t\in H_{2}}2^{2n+1}:H_{1}<H_{2}\right\} $$ is $IP_{\infty}$-set but not $IP$-set. As a consequences $A$ is  $IP^{\star}$-set which is not $IP^{2\star}$-set  and at the same time,  not an $IP_{r}^{\star}$-set for any $r\in\mathbb{N}$.
 \end{theorem}
 
 \begin{proof}
 	Let $A=\left\{ \sum_{t\in H}2^{2n}+\sum_{t\in K}2^{2n+1}:H<K\right\}$ be an $IP$-set. Then there exists an sequence $\langle x_{n}\rangle_{n=1}^{\infty}$ and $ FS\left(\langle x_{n}\rangle_{n=1}^{\infty}\right)\subset A$. Let $$ x_{1}=\sum_{n\in H_{1}}2^{2n}+\sum_{n\in K_{1}}2^{2n+1}  \text{ with } H_{1}<K_{1}.$$ \\  
 \textbf{Case-I}:	As $ \langle x_{n}\rangle_{n=1}^{\infty}$ is infinite , we can find for some $k\in\mathbb{N}$ $$x_{k}=\sum_{n\in H_{k}}2^{2n}+\sum_{n\in K_{k}}2^{2n+1} \text{ with }  H_{k}<K_{k}.$$  Now  $$ x_{1}+x_{k}=\sum_{n\in H_{1}\cup H_{k}}2^{2n}+\sum_{n\in K_{1}\cup K_{k}}2^{2n+1}  \text{ but }  H_{1}\cup H_{k}\nless K_{1}\cup K_{k}. $$ Hence $ x_{1}+x_{k}\notin A$.\\
 
 \textbf{Case-II}:	As $\left\{ 1,2,\ldots,\max H_{2}\right\}  $ is finite set, we can find $$ x_{s}=\sum_{n\in H_{s}}2^{2n}+\sum_{n\in K_{s}}2^{2n+1}  \text{ with }  H_{s}<K_{s}.$$ and $$ x_{t}=\sum_{n\in H_{t}}2^{2n}+\sum_{n\in K_{t}}2^{2n+1} \text{ with }  H_{t}<K_{t}.$$ Where $H=H_{s}=H_{t}$  and $K_{t}<K_{s}$. In this case $ x_{s}+x_{t}=\sum_{n\in H\cup K_{t}\cup K_{s}}2^{2n+1}$ is not in $A$.\\
 	Let $ r\in\mathbb{N} $, $ x_{i}=2^{2i}+2^{2\left(r+1\right)+i} $ for $ i=1,2,\ldots,r$.  Then $FS\left(\langle x_{n}\rangle_{n=1}^{r}\right)\subset A$ and $ A $ is an $IP_{\infty}$-set. 
 \end{proof}

 \section{Zigzag finite sums and products family}
 
We start this section with the following definition:
 \begin{definition}\label{ ZFSP-family}
 	 Let $\mathcal{F}\subseteq\mathcal{P}\left(\mathbb{N}\right)\setminus\left\{ \emptyset\right\}  $ and $ \mathcal{F} $ is called $ ZFSP$-family, if $ \mathcal{F} $ satisfies the following properties
 	\begin{itemize}
 		\item[(a)] For each $ A\in\mathcal{F}$, $A$ is an $ IP^{\star} $-set in $ \left(\mathbb{N},+\right)$
 		\item[(b)] For each $ A\in\mathcal{F} $ , there exists $ C\in\mathcal{F}$ such that, $ C\subseteq A$ and for $ n\in C, -n+C\in\mathcal{F}$
 		\item[(c)] For $ A,B\in\mathcal{F}$, $A\cap B\in\mathcal{F}$
 		\item[(d)] For $ n\in\mathbb{N}$ and $A\in\mathcal{F}$, $n^{-1}A\in\mathcal{F}$
 	\end{itemize}
Let $\mathcal{F}_{1}=\left\{\textbf{ collection of all dynamical } IP^{\star }\textbf{-sets }\right\}$ forms a  $ZFSP$ family by \cite{G}. By \cite{LL}, $\mathcal{F}_{2}=\left\{\textbf{ collection of all topological dynamical } IP^{\star }\textbf{-sets }\right\}$ forms a  $ZFSP$ family. The following theorem shows that any member of a $ZFSP$ family  contains zigzag finite sums and products  of a sum subsystem of any finite number of sequences.	 
 \end{definition}
 
 \begin{theorem}\label{Zig-zag in ZFSP}
 	 Let $ l\in\mathbb{N}$ and $A\subseteq\mathbb{N}$ is in a $ZFSP$ family set in $ \left(\mathbb{N},+\right)$. Then for any $ l $ sequences $ \langle x_{n}^{\left(1\right)}\rangle_{n=1}^{\infty},\langle x_{n}^{\left(2\right)}\rangle_{n=1}^{\infty},\ldots,\langle x_{n}^{\left(l\right)}\rangle_{n=1}^{\infty}$ in $ \mathbb{N}$ there exists a $ l $sum subsystems $ \langle y_{n}^{\left(i\right)}\rangle_{n=1}^{\infty}$ of $ \langle x_{n}^{\left(i\right)}\rangle_{n=1}^{\infty} $ for each $ i\in\left\{ 1,2,\ldots,l\right\}  $ such that $$ZFS\left(\langle y_{n}^{\left(i\right)}\rangle_{i,n=1,1}^{l,\infty}\right)\bigcup ZAP\left(\langle y_{n}^{\left(i\right)}\rangle_{i,n=1,1}^{l,\infty}\right)\subset A.$$
 \end{theorem}
 
 \begin{proof}
 	 Let $\mathcal{F}$ be the given $ZFSP$ family. Form the definition of $ZFSP$ family, We know from that if $A\in\mathcal{F}$, there exists $ C\in\mathcal{F}$ such that, $C\subseteq A$ and for $n\in C$, $ -n+C\in\mathcal{F}$.
 	
 	 Choose $ \langle H_{1}^{\left(i\right)}\rangle_{i=1}^{l}\subset\mathcal{P}_{f}\left(\mathbb{N}\right)$ such that $ \sum_{t\in H_{1}^{\left(i\right)}}x_{t}^{\left(i\right)}=y_{1}^{\left(i\right)}\in C$ for each $ i\in\left\{ 1,2,\ldots,l\right\} $. Let us assume for $ m\in\mathbb{N}$, we have a sequence $ \langle y_{n}^{\left(i\right)}\rangle_{n,i=1,1}^{m,l}$ in $\mathbb{N}$ and $ \langle H_{n}^{\left(i\right)}\rangle_{n,i=1,1}^{m,l}$ in $\mathcal{P}_{f}\left(\mathbb{N}\right)$ such that
 	\begin{itemize}
 		\item[1] For each $ j\in\left\{ 1,2,\ldots,m-1\right\}  $ and $ i\in\left\{ 1,2,\ldots,l\right\}  $ $\max H_{j}^{\left(i\right)}<\min H_{j+1}^{\left(i\right)}$.
 		\item[2] If for each $j\in\left\{ 1,2,\ldots,m\right\} $, $ y_{j}^{\left(i\right)}=\sum_{t\in H_{j}^{\left(i\right)}}x_{t}^{\left(i\right)} $
 	\end{itemize}
 		 then,
 		 $$ B=ZFS\left(\langle y_{j}^{\left(i\right)}\rangle_{j,i=1,1}^{m,l}\right)\bigcup ZFP\left(\langle y_{j}^{\left(i\right)}\rangle_{j,i=1,1}^{m,l}\right)\subset C$$
 		 Now we have, $$D=C\cap\bigcap_{n\in E_{1}}\left(-n+C\right)\cap\bigcap_{n\in E_{2}}\left(n^{-1}C\right)$$ is in $\mathcal{F}$ . Now choose a sequence $\langle H_{m+1}^{\left(i\right)}\rangle_{i=1}^{l}$ such that $ \max H_{m}^{\left(i\right)}<\min H_{m+1}^{\left(i\right)}$ and $y_{m+1}^{\left(i\right)}=\sum_{t\in H_{m+1}^{\left(i\right)}}x_{t}^{\left(i\right)}\in A$, $ i\in\left\{ 1,2,\ldots,l\right\} $. Then one can now check this choice of $ \langle y_{m+1}^{\left(i\right)}\rangle_{i=1}^{l}$ completes the induction. So we have the result
 	
 \end{proof}
 
Let us recall the dynamically $IP^{\star}$-sets. A subset $C$ of $ \mathbb{N}$ is dynamical $ IP^{\star}$- set iff there exist a measure preserving system $ \left(X,\mathcal{B},\mu,T\right)$ and an $ A\in \mathcal{B}$ with $ \mu\left(A\right)>0$ such that $\left\{ n\in\mathbb{N}:\mu\left(A\cap T^{-n}A\right)>0\right\} \subseteq C$. Now, if we cosider $ A,B\in\ \mathcal{B}$ with $ \mu\left(A\right)\mu\left(B\right)>0$ then the $ E=\left\{ n\in\mathbb{N}:\mu\left(A\cap T^{-n}B\right)>0\right\} $ may not be an $ IP^{\star}$-set, even $E$ may be an empty. For example, if we cosider a measure preserving system $ \left(X,\mathcal{B},\mu,I\right) $ where $ I $ is identity transformation on $X $. Then for any $A,B\in\mathcal{B} $, with $A\cap B=\emptyset$ we have $\left\{ n\in\mathbb{N}:\mu\left(A\cap I^{-n}B\right)\right\} =\emptyset$.

\begin{definition}
	 A measure preserving system $\left(X,\mathcal{B},\mu, T\right)$ is called mild mixing, if for any $\epsilon>0$ and $A,B\in\mathcal{B}$ with $\mu\left(A\right)\mu\left(B\right)>0$ $$\left\{ n:|\mu(A\cap T^{-n}B)-\mu(A)\mu(B)|<\epsilon\right\} $$ is an $IP^{\star}$-set.
\end{definition}
From the above definition, it is clear that for a mild mixing $\left(X,\mathcal{B},\mu, T\right)$ for $ A,B\in\mathcal{B}$ with $ \mu\left(A\right)\mu\left(B\right)>0$, $ \left\{ n\in\mathbb{N}:\mu\left(A\cap T^{-n}B\right)>0\right\} $ is an $IP^{\star}$-set

\begin{definition}
	 Let $C$ be a subset of $\mathbb{N}$, is called $ MIP^{\star}$-sets iff there exist a mild mixing system $ \left(X,\mathcal{B},\mu,T\right)$ and $A,B\in\mathcal{B}$ with $\mu\left(A\right)\mu\left(B\right)>0$ such that $\left\{ n\in\mathbb{N}:\mu\left(A\cap T^{-n}B\right)>0\right\} \subseteq C$.
\end{definition}

From \cite[Proposition 9.23]{F}, we get the following:
\begin{lemma}\label{product of mild mixing}
	Product of two mild mixing systems is a mild mixing.
\end{lemma}

\begin{lemma}\label{power of mild mixing}
	Let $ \left(X,\mathcal{B},\mu,T\right)$ be a mild mixing system. Then for any $ m\in\mathbb{N}$, $ \left(X,\mathcal{B},\mu,T^{m}\right)$ is also a mild mixing system.
\end{lemma}
\begin{proof}
	For any $\epsilon>0$ and $A,B\in\mathcal{B}$ with $\mu\left(A\right)\mu\left(B\right)>0$,  we have to show that $\left\{ n:|\mu(A\cap T^{-mn}B)-\mu(A)\mu(B)|<\epsilon\right\} $ is an $IP^{\star}$-set. Now $E=\left\{ n:|\mu(A\cap T_{n}^{-1}B)-\mu(A)\mu(B)|<\epsilon\right\} $ is an $ IP^{\star}$-set, which implies $ m^{-1}E=\left\{ n:|\mu(A\cap T^{-mn}B)-\mu(A)\mu(B)|<\epsilon\right\} $ and we get the required result as, $ m^{-1}E$ is an $IP^{\star}$-set.
\end{proof}

\begin{theorem}
 The measure preserving dynamical system $\left(X,B,\mu,T\right)$ is mild mixing iff for any $A_{0},A_{1},\ldots\,A_{k} \in\mathcal{B}$ with $\mu\left(A_{0}\right)\mu\left(A_{1}\right)\cdots\mu\left(A_{k}\right)>0$ and $n_{1},\ldots,n_{k}\in\mathbb{N}$ with $0=e_{0}<e_{1}<\ldots<e_{k}$ any $p\in E \left(\beta \mathbb{N}\right)$, $$p\text{-}\lim_{n}\mu\left(\bigcap_{i=0}^{k}T^{-ne_{i}}A_{i}\right)=\prod_{i=0}^{k}\mu\left(A_{i}\right).$$
\end{theorem}
\begin{proof}
	Follows from \cite[Theorem 9.27, Page-194]{F}
\end{proof}

From the above mentioned theorem,  it is an immediate consequence that for a mild mixing system  $\left(X,B,\mu,T\right)$ and for any $A_{0},A_{1},\ldots\,A_{k} \in\mathcal{B}$ with $\mu\left(A_{0}\right)\mu\left(A_{1}\right)\cdots\mu\left(A_{k}\right)>0$, the set $\left\{n:\mu\left(\bigcap_{i=0}^{k}T^{-in}A_{i}\right)>0\right\}$ is an $IP^{\star}$-set. From this we define $k$th order mild mixing $IP^{\star}$-set, abbreviated by  $k\text{-}MIP^{\star}$-set.

\begin{definition}
	 Let $C \subset\mathbb{N}$, is called $k\text{-}MIP^{\star}$-set iff there exist a mild mixing system $ \left(X,\mathcal{B},\mu,T\right)$ and $A_{0},A_{1},\ldots\,A_{k} \in\mathcal{B}$ with $\mu\left(A_{0}\right)\mu\left(A_{1}\right)\cdots\mu\left(A_{k}\right)>0$ such that $\left\{n:\mu\left(\bigcap_{i=0}^{k}T^{-in}A_{i}\right)>0\right\} \subseteq C$.
\end{definition}

Now we observe some properties of $k\text{-}MIP^{\star}$-sets.

\begin{lemma}
Let $m,k\in \mathbb{N}$. Let $A$ be a $k\text{-}MIP^{\star}$-set. Then $m^{-1}A$ is also a $k\text{-}MIP^{\star}$-set.
\end{lemma}
\begin{proof}
As $A$ is $k\text{-}MIP^{\star}$-set, there exist a mild mixing system $ \left(X,\mathcal{B},\mu,T\right)$ and $A_{0},A_{1},\ldots\,A_{k} \in\mathcal{B}$ with $\mu\left(A_{0}\right)\mu\left(A_{1}\right)\cdots\mu\left(A_{k}\right)>0$ such that $$\left\{n:\mu\left(\bigcap_{i=0}^{k}T^{-in}A_{i}\right)>0\right\} \subseteq A.$$ Now if $E=\left\{n:\mu\left(\bigcap_{i=0}^{k}T^{-in}A_{i}\right)>0\right\}$, then, $$m^{-1}E=\left\{n:\mu\left(\bigcap_{i=0}^{k}T^{-imn}A_{i}\right)>0\right\} \subseteq m^{-1}A.$$ Now $m^{-1}A$ is a $k\text{-}MIP^{\star}$-set for $ \left(X,\mathcal{B},\mu,T^{m}\right)$  being a mild mixing system by Lemma \ref{power of mild mixing}.
\end{proof}

\begin{lemma}
	Let $k\in \mathbb{N}$. Let $A$ and $B$ be  $k\text{-}MIP^{\star}$-sets. Then their intersection $A\cap B$ is also a  $k\text{-}MIP^{\star}$-set.
\end{lemma}

\begin{proof}
 There exist two mild mixing system $ \left(X,\mathcal{A},\mu,T\right)$ and $ \left(X,\mathcal{B},\nu,S\right)$ and $A_{0},A_{1},\ldots\,A_{k} \in\mathcal{A}$ with $\mu\left(A_{0}\right)\mu\left(A_{1}\right)\cdots\mu\left(A_{k}\right)>0$ and $B_{0},B_{1},\ldots\,B_{k} \in\mathcal{B}$ with $\nu\left(B_{0}\right)\nu\left(B_{1}\right)\cdots\nu\left(B_{k}\right)>0$  such that $$\left\{n:\mu\left(\bigcap_{i=0}^{k}T^{-in}A_{i}\right)>0\right\} \subseteq A \text{ and }\left\{n:\nu\left(\bigcap_{i=0}^{k}S^{-in}B_{i}\right)>0\right\} \subseteq B.$$ 
  Now by the lemma \ref{product of mild mixing},  $ \left(X\times Y,\mathcal{A}\times \mathcal{B},\mu \times \nu,T\times S\right)$  is mild mixing system. Here $\mathcal{A}\times\mathcal{B}$ is a $\sigma$-algebra generated by $\left\{A\times B:A\in\mathcal{A} \text { and } B\in \mathcal{B}\right\}$ with $\mu\times\nu\left(A\times B\right)=\mu\left(A\right)\nu\left(B\right), \forall A\in\mathcal{A}\text{ and }\forall B\in\mathcal{B}$.Then for some $n\in \mathbb{N}$, $\mu\times \nu \left(\bigcap_{i=0}^{k}\left(T\times S\right)^{-in}\left(A_{i}\times B_{i}\right)\right)>0$ if and only if $\mu\left(\bigcap_{i=0}^{k}T^{-in}A_{i}\right)>0$ and $\nu\left(\bigcap_{i=0}^{k}S^{-in}B_{i}\right)>0$.Then $$\left\{n: \mu\times \nu \left(\bigcap_{i=0}^{k}\left(T\times S\right)^{-in}\left(A_{i}\times B_{i}\right)\right)>0\right\}\subseteq A\cap B$$  and $ \left(X\times Y,\mathcal{A}\times \mathcal{B},\mu \times \nu,T\times S\right)$ being a mild mixing system, we have the  intersection $A\cap B$ is a  $k\text{-}MIP^{\star}$-set.
\end{proof}

\begin{lemma}
	Let $n\in \mathbb{N}$. Let $A$ be a $k\text{-}MIP^{\star}$-set.Then there exists a $k\text{-}MIP^{\star}$-set $B$ such that $B\subset A$ with for all  $n\in B$, $-n+B$ is a $k\text{-}MIP^{\star}$-set.
\end{lemma}

\begin{proof}
As $A$ is $k\text{-}MIP^{\star}$-set, there exist a mild mixing system $ \left(X,\mathcal{B},\mu,T\right)$ and $A_{0},A_{1},\ldots\,A_{k} \in\mathcal{B}$ with $\mu\left(A_{0}\right)\mu\left(A_{1}\right)\cdots\mu\left(A_{k}\right)>0$ such that $$ B=\left\{n:\mu\left(\bigcap_{i=0}^{k}T^{-in}A_{i}\right)>0\right\} \subseteq A.$$ To see that $B$ is required , let $n\in B$ and $C=\bigcap_{i=0}^{k}T^{-in}A_{i}$. We claim that $ \left\{n:\mu\left(\bigcap_{i=0}^{k}T^{-in}C\right)>0\right\} \subseteq -n+B $.  Let for some $m\in \mathbb{N}$, we have $\mu\left(\bigcap_{i=0}^{k}T^{-im}C\right)>0$ . Now $\bigcap_{i=0}^{k}T^{-im}C=\bigcap_{i=0}^{k}T^{-im}\left(\bigcap_{j=0}^{k}T^{-jn}A_{j}\right)$,  which is a subset of $\bigcap_{i=0}^{k}T^{-i\left(n+m\right)}A_{i}$. Then $ m\in {-n+B}$ as  $\mu\left(\bigcap_{i=0}^{k}T^{-i\left(n+m\right)}A_{i}\right)>0$.

\end{proof}
  
 \begin{theorem}\label{MIP-ZFPS}
 	 Let $k\in \mathbb{N}$. The collection of all  $k\text{-}MIP^{\star}$-sets forms a $ZFSP$ family.
 \end{theorem}
 \begin{proof}
 	Follows from the above lemmas.
 \end{proof}
In \cite{FK}, Furstenberg and Katznelson proved the following:
\begin{theorem}
 Let $k\in \mathbb{N}$.	Let $ \left(X,\mathcal{B},\mu,T\right)$ be a measure preserving system  and $A\in \mathcal{B}$ with $\mu\left(A\right)>0$. Then there exists $r\in \mathbb{N}$ such that $$R_{k} =\left\{n:\mu \left(\bigcap_{i=0}^{k}T^{-in}A\right)>0\right\}.$$ is an $IP_{r}^{\star}$-set.
\end{theorem}
  
  \begin{definition}
  	Let $k\in \mathbb{N}$. Let $C$ be a subset of $\mathbb{N}$, is called $k\text{-}RIP^{\star}$-set iff there exist a measure preserving system $ \left(X,\mathcal{B},\mu,T\right)$ and $A\in \mathcal{B}$ with $\mu\left(A\right)>0$ such that $\left\{n:\mu\left(\bigcap_{i=0}^{k}T^{-in}A\right)>0\right\} \subseteq C$.
  	
  \end{definition}
  
   Like the family of $k\text{-}MIP^{\star}$-sets, we can prove that the collection of $k\text{-}RIP^{\star}$-sets forms a $ZFSP$ family. So it is a simple consequence that any $k\text{-}RIP^{\star}$-set is an $IP^{n\star}$-set and  states the following:
  
\begin{corollary}
	Let $k,n\in \mathbb{N}$.	Let $ \left(X,\mathcal{B},\mu,T\right)$ be a measure preserving system  and $A\in \mathcal{B}$ with $\mu\left(A\right)>0$. Then $$R_{k} =\left\{n:\mu \left(\bigcap_{i=0}^{k}T^{-in}A\right)>0\right\}$$ is an $IP^{n\star}$-set.
\end{corollary}
  
  For a mild mixing system, we get the following corollary applying Theorem\ref{MIP-ZFPS} and Theorem\ref{Zig-zag in ZFSP}:
  \begin{corollary}
  	Let $k,n\in \mathbb{N}$.	Let $ \left(X,\mathcal{B},\mu,T\right)$ be a mild mixing. Then  $\left\{n:\mu \left(\bigcap_{i=0}^{k}T^{-in}A_{i}\right)>0\right\}$ is an $IP^{n\star}$-set, for any $A_{0},A_{1},\ldots\,A_{k} \in\mathcal{B}$ with $\mu\left(A_{0}\right)\mu\left(A_{1}\right)\cdots\mu\left(A_{k}\right)>0$.
  \end{corollary}
  
\bibliographystyle{plain}

\begin{thebibliography}{9}
	
	\bibitem{BH} V. Bergelson and N. Hindman, Quotient sets and density recurrent sets, Trans. Amer. Math. Soc. 364 (2012), 4495-4531.
	
	\bibitem{BH2} V. Bergelson and N. Hindman, On $IP^{\star}$-sets and central sets, Combinatorica 14 (1994), 269-277.
	
    \bibitem{D1} D. De, Additive and Multiplicative structure of , $C^{\star}$-set Integers, 14, 2(2014), A26
	
     \bibitem{D2} D. De, Combined algebraic properties of central ${}^{\star}$- sets, Integers 7 (2007), A37.
	
    \bibitem{DHS} D. De, N. Hindman, and D. Strauss, A new and stronger Central Sets Theorem, Fundamenta Mathematicae 199 (2008), 155-175.
	
	\bibitem{F} H. Furstenberg, Recurrence in Ergodic Theory and Combinatorial Number Theory, Princeton University Press, 1981.
	
	\bibitem{FK} H. Furstenberg and Y. Katznelson, An ergodic Szemerédi theorem for IP-systems and combinatorial theory, d’Analyse Mathématique , Volume 45, pages 117–168, (1985).
 
    \bibitem{G1} S. Goswami, Product of difference sets of the set of primes, Proc. Amer. Math. Soc. 151 (2023) 5081-5086.
 
	\bibitem{G}	S. Goswami, Combined Zigzag structure in Dynamically $IP^{\star}$-sets, Topology and its Applications, volume 300, 15th August 2021, 107752.

   \bibitem{GHW} S. Goswami, W.Huang, X.Wu, on the set of kornecker number.Bulletine of the Australian Mathematical socity. Published online 2024: 1-9.
	
	\bibitem{H} N. Hindman: Finite sums from sequences within cells of partitions of $\mathbb{N}$, J. Comb. Theory (Series A) 17 (1974), 1-11.
	
	\bibitem{HS} N. Hindman and D. Strauss, Algebra in the	Stone-\v Cech compactification: theory and applications, second edition, de Gruyter, Berlin, 2012.
	
	\bibitem{LL} X. Liang and Q. Liao, Characterizations of topological dynamical IP* sets, Journal of Difference Equations and Applications, (2024) 1–24.
		
	\bibitem{Z1} T. Zhang, Zigzag structures in $IP^{\star}$ sets and dynamical $IP^{\star}$ sets, Topology and its Applications, Volume 327, 15 March 2023, 108437.
	
	\bibitem{Z2} T. Zhang,  Zigzag structures in $IP^{\star}$ sets, Semigroup Forum , 25 May 2023, Volume 106, pages 747–750, (2023).
	
	

	
\end{thebibliography}

 \end{document}